\documentclass[12pt]{article}
\usepackage{amsmath,amssymb, amsthm}
\usepackage{hyperref}

\def\R{{\mathbb R}}
\def\C{{\mathbb C}}
\def\CC{{\mathcal C}}

\def\CP{{\mathcal P}}
\def\CS{{\mathcal S}}

\def\re{\text{Re}}
\def\im{\text{Im}}
\def\degbb{\deg_{\text{BB}}}

\newtheorem{theorem}{Theorem}
\newtheorem{lemma}[theorem]{Lemma}
\newtheorem{corollary}[theorem]{Corollary}
\newtheorem{definition}[theorem]{Definition}
\theoremstyle{definition}

\begin{document}

\renewcommand\footnotemark{}
\title{Incidence bounds for complex algebraic curves on Cartesian products\thanks{The first author was supported by ERC Advanced Research Grant
no. 267165 (DISCONV), by Hungarian National Research Grant NK 104183, and by NSERC.
The second author was partially supported by Swiss National Science Foundation Grants 200020-144531 and 200021-137574.
Part of this research was performed while the authors visited the Institute for Pure and Applied Mathematics (IPAM) in Los Angeles, which is supported by the National Science Foundation.}}

\author{J\'ozsef Solymosi \and Frank de Zeeuw}

\date{}
\maketitle
\begin{abstract}
We prove bounds on the number of incidences between a set of algebraic curves in $\C^2$ and a Cartesian product $A\times B$ with finite sets $A,B\subset \C$.
Similar bounds are known under various restrictive conditions, but we show that the Cartesian product assumption leads to a simpler proof and lets us remove these conditions. 
This assumption holds in a number of interesting applications, and with our bound these applications can be extended from $\R$ to $\C$.
We also obtain more precise information in the bound, which is used in several recent papers \cite{RSZ, VZ}.
Our proof works via an incidence bound for surfaces in $\R^4$, which has its own applications \cite{RS}.
The proof is a new application of the polynomial partitioning technique introduced by Guth and Katz \cite{GK}.
\end{abstract}

\medskip

\section{Introduction}
Not many incidence bounds have been proved over the complex numbers.
The quintessential incidence bound of Szemer\'edi and Trotter for points and lines in $\R^2$ was generalized to $\C^2$ by T\'oth \cite{T} and Zahl \cite{Z}. 
It states that for a finite set $P$ of points in $\C^2$ and a finite set $L$ of lines in $\C^2$, the set of \emph{incidences}, denoted by $I(P,L) := \{(p,\ell)\in P\times L: p\in\ell\}$, satisfies
\[|I(P,L)| = O\left(|P|^{2/3}|L|^{2/3}+|P|+|L|\right).\]
The Szemer\'edi-Trotter bound was generalized to algebraic (and even continuous) curves in $\R^2$ by Pach and Sharir \cite{PS}, but their result has not yet been fully extended to $\C^2$. 
Solymosi and Tao \cite{ST} and Zahl \cite{Z} did prove complex versions, but only for algebraic curves satisfying certain restrictions.
These restrictions include the requirement that the curves are smooth and that the intersections of the curves are transversal (i.e., the curves have distinct tangent lines at their intersection points); both restrictions do not hold in many potential applications.
Concurrently with this paper, Sheffer and Zahl \cite{SZ} proved a complex version of the Pach-Sharir bound without such restrictions, but with a slightly weaker bound.

Incidence bounds are often easier to prove when the point set has the structure of a Cartesian product.
This obervation was used by Solymosi and Vu \cite{SV} to obtain incidence bounds in $\R^D$.
It was noted by Solymosi \cite{S} that the Szemer\'edi-Trotter bound in $\C^2$ can be proved more easily (compared to \cite{T}) when the point set is a Cartesian product; this statement was then used to obtain a sum-product bound over $\C$.
Solymosi and Tardos \cite{ST} used the same observation to obtain bounds on rich M\"obius transformation from $\C$ to $\C$.

We prove a Pach-Sharir-like incidence bound for algebraic curves in $\C^2$, under the assumption that the point set is a Cartesian product $A\times B$ with $A,B\subset \C$.
We do not require the curves to satisfy the restrictions that were needed in \cite{ST,Z}, 
and the bound is slightly stronger than in \cite{SZ}.
Like in \cite{PS,ST,Z, SZ}, the curves must satisfy a degrees-of-freedom condition, which can come in different forms.
Theorem \ref{thm:mainintro} states our main result with what is probably the most convenient condition; several other versions can be found in Section \ref{sec:corcurves}.

\begin{theorem}\label{thm:mainintro}
Let $A$ and $B$ be finite subsets of $\C$ with $|A|=|B|$, and set $\CP := A\times B$.
Let $\CC$ be a finite set of algebraic curves in $\C^2$ of degree at most $d$, such that any two points of $\CP$ are contained in at most $M$ curves of $\CC$.
Then
$$I(\CP, \CC) 
= O(d^{4/3}M^{1/3}|\CP|^{2/3}|\CC|^{2/3}
+M(\log M+\log d) |\CP|+ d^4|\CC|).$$
\end{theorem}

We have worked out in detail the dependence of the bound on the parameters, which is of interest in certain applications, in particular \cite{RSZ, VZ}.
Our proof works via an incidence bound for well-behaved surfaces in $\R^4$, which is interesting in its own right; it was used by Raz and Sharir \cite{RS} to improve the best known bound on the number of unit area triangles determined by a point set in $\R^2$.

Although the assumption that the point set is a Cartesian product is very restrictive, it is satisfied in a number of interesting problems.
We give several examples of such applications in Section \ref{sec:applications}.
These include an answer to a question of Elekes \cite{E02} related to sum-product estimates, and a generalization to $\C$ of a recent result of Sharir, Sheffer, and Solymosi \cite{SSS} on distinct distances between lines.
More sophisticated applications can be found in the already mentioned works \cite{RS,RSZ,VZ}, which were in fact the original motivation for this paper.

We begin in Section \ref{sec:real} with the elementary proof of the real analogue of our main theorem, which is not a new result, but serves as an introduction to our main proof.
In Section \ref{sec:tools} we collect the technical tools that we use, and in Section \ref{sec:main} we prove our main bound, which concerns point-surface incidences in $\R^4$.
In Section \ref{sec:corsurfaces} we deduce some corollaries for surfaces in $\R^4$, 
and in Section \ref{sec:corcurves} we prove corollaries for curves in $\C^2$, including Theorem \ref{thm:mainintro} above.
Finally, in Section \ref{sec:applications}, we give three applications.

\section{Warmup: Points and curves in \texorpdfstring{$\R^2$}{the real plane}.}
\label{sec:real}

As a warmup for the complex case, we first give a proof of the corresponding statement for incidences between real algebraic curves and a Cartesian product in $\R^2$.
This is not a new result, as it follows from the work of Pach and Sharir \cite{PS}.
The proof given here is, however, much simpler, because the product structure allows for a trivial partitioning of the plane; the proof is self-contained up to a few basic facts about algebraic curves.
Moreover, it provides a blueprint for our main proof in Section \ref{sec:main}.

Throughout, given a set $\CP$ of points and a set $\CC$ of geometric objects, we define the set of \emph{incidences} by $I(\CP,\CC):=\{(p,c)\in \CP\times\CC: p\in c\}$.

\begin{theorem}\label{thm:warmup}
Let $A$ and $B$ be finite subsets of $\R$ and $\CP := A\times B$.
Let $\CC$ be a finite set of algebraic curves in $\R^2$ of degree at most $d$ such that any two points of $\CP$ are contained in at most $M$ curves of $\CC$.
We assume that no curve in $\CC$ contains a horizontal or vertical line, that $d^4|\CC| \leq M|\CP|^2$,
and that $|A|\leq|B|$ and $d|\CC|\geq M|B|^2/|A|$.
Then
$$|I(\CP, \CC)| = O(d^{2/3}M^{1/3}|\CP|^{2/3}|\CC|^{2/3}).$$
\end{theorem}
\begin{proof}
Let $r$ be a real number, to be chosen at the end of the proof, satisfying $d\leq r\leq |A|$.
We ``cut'' $\R$ in $O(r)$ points that are not in $A$, 
splitting $\R$ into $O(r)$ intervals so that each interval contains $O(|A|/r)$ elements  of $A$ (this is possible because $r\leq |A|$).
Similarly, we choose $O(r)$ cutting points not in $B$ that split $\R$ into at most $O(r)$ intervals, each containing $O(|B|/r)$ elements of $B$ (using $r\leq |A|\leq |B|$). 
This gives a partition of $\R^2$ into $O(r^2)$ open cells (which are rectangles) and a closed boundary (consisting of $O(r)$ lines). Each cell contains $O(|A||B|/r^2)=O(|\CP|/r^2)$ points of $\CP$,
while the boundary is disjoint from $\CP$.

We need to bound the number of cells that a curve $C\in\CC$ can intersect.
The curve has $O(d^2)$ connected components by Harnack's Theorem (see Lemma \ref{lem:baronebasu} below),
and it has at most $d$ intersection points with each of the $O(r)$ boundary lines by B\'ezout's Inequality (see Lemma \ref{lem:bezout} below), using the fact that the curve contains no horizontal or vertical line.
Thus the $O(d^2)$ connected components are cut in $O(dr)$ points.
By wiggling the cutting lines slightly, we can ensure that they do not hit a curve of $\CC$ in a singularity, since algebraic curves have finitely many singularities (see Section \ref{sec:tools}).
Therefore, each cut increases the number of connected components by at most one.\footnote{This could fail if a cutting point were a singularity (although even then the number of branches could be controlled with some more effort).}
Thus any $C\in \CC$ intersects $O(d^2+d r)= O(dr)$ (using $d\leq r$) of the $O(r^2)$ cells.\footnote{This fact can be obtained more directly using Lemma \ref{lem:baronebasu} below, by noting that the union of the lines is a curve defined by a polynomial $f$ of degree $O(r)$, so $C\backslash Z(f)$ has $O(dr)$ connected components.
However, we have used the argument above because it will play a crucial role in the proof of our main theorem.}

Let $I_1$ be the subset of incidences $(p,C)\in I(\CP,\CC)$ such that $(p,C)$ is the only incidence of $C$ in the cell containing $p$,
and let $I_2$ be the subset of incidences $(p,C)\in I(\CP,\CC)$ such that $C$ has at least one other incidence in the cell that contains $p$.
Then, since a curve intersects $O(dr)$ cells, we have
$$|I_1| =O(dr|\CC|).$$
On the other hand, given two points in one cell, there are by assumption at most $M$ curves in $\CC$ that contain both points.
Thus, in a cell with $k$ points there are at most $2M\binom{k}{2} = O(Mk^2)$ incidences from $I_2$.
Therefore, summing over all cells, we have
$$|I_2| = O\left(r^2\cdot M\left(\frac{|\CP|}{r^2}\right)^2\right) = O\left( \frac{M|\CP|^2}{r^2}\right).$$
Choosing $r^3 := \frac{M}{d}\frac{|\CP|^2}{|\CC|}$ 
gives
\begin{align*}
|I(\CP, \CC )|=I_1 + I_2 =O\left(d^{2/3}M^{1/3} |\CP|^{2/3}|\CC|^{2/3}\right).
\end{align*}

We have to verify that our choice of $r$ satisfies $d \leq r$ and $r\leq |A|$.
The first follows from the assumption $d^4|\CC| \leq M|\CP|^2$ and
$$r =  \left(\frac{M|\CP|^2}{d|\CC|}\right)^{1/3}
\geq \left(\frac{d^4|\CC|}{d |\CC|}\right)^{1/3}
= d.
$$
The second follows from the assumption $d|\CC|\geq  M|B|^2/|A|$ and
$$r^3 = \frac{M|\CP|^2}{d|\CC|}
\leq \frac{M|\CP|^2}{M|B|^2/|A|}
= |A|^3.$$
This completes the proof.
\end{proof}

Theorem \ref{thm:warmup} has several conditions which can be simplified in various ways to make the statement more suitable for application.
We state one version here as an example,
but we refer to Section \ref{sec:corcurves} for the proof
(which is identical to that of the complex version presented there).

\begin{corollary}\label{cor:realpractical}
Let $A$ and $B$ be finite subsets of $\R$ with $|A|= |B|$, and $\CP := A\times B$.
Let $\CC$ be a finite set of algebraic curves in $\R^2$ of degree at most $d$, such that any two points of $\CP$ are contained in at most $M$ curves of $\CC$.
Then
$$|I(\CP, \CC)| 
= O\left(d^{2/3}M^{1/3}|\CP|^{2/3}|\CC|^{2/3}
+ M(\log M+\log d)|\CP|
+ d^2|\CC|
\right).$$
\end{corollary}

\section{Definitions and tools}
\label{sec:tools}

\subsection{Definitions}\label{subsec:defs}
We introduce a few definitions and basic facts from algebraic geometry in some detail, because the subtle differences between real and complex varieties play a role in our proof.

A \emph{variety} in $\C^D$ is a set of the form 
$$Z_{\C^D}(f_1,\dots,f_m) 
:= \left\{(z_1,\ldots,z_D)\in \C^D: f_i(z_1,\ldots, z_D)=0 ~\text{for}~ i=1,\ldots,m\right\},$$
for polynomials $f_i\in \C[z_1,\ldots,z_D]$.
Such sets are normally called \emph{affine} varieties (or just \emph{zero sets}), but since this is the only type of variety that we consider, we refer to them simply as varieties.
Similarly, we define a \emph{real variety}
to be a zero set of the form 
$$Z_{\R^D}(f_1,\ldots,f_m):= \left\{(x_1,\ldots,x_D)\in \R^D: f_i(x_1,\ldots, x_D)=0 ~\text{for}~ i=1,\ldots,m\right\}$$ 
with polynomials $f_i\in \R[x_1,\ldots,x_D]$.
We refer to \cite{Ha,He} for definitions of the dimension $\dim_\C(V)$ and the degree $\deg(V)$ of a complex variety $V$, and we refer to \cite[Section 5.3]{BPR} or \cite[Section 2.8]{BCR} for a careful definition of the real dimension of a real variety $W$, denoted by $\dim_\R(W)$. 
One can locally view a real variety as a real manifold (around any nonsingular point, see below), and the real dimension equals the dimension in the manifold sense (more precisely, it is the \emph{maximum} dimension at any nonsingular point).

A \emph{complex algebraic curve} in $\C^2$ is a variety $V$ with $\dim_\C(V) = 1$.
In our definition\footnote{Note that the dimension of a reducible variety is the maximum of the dimensions of its components, so a curve can have zero-dimensional components.
The degree of a reducible variety is the sum of the degrees of its components, so a curve of degree $d$ with $k$ zero-dimensional components has a purely one-dimensional component of degree $d-k$.}, a curve $C\subset \C^2$ of degree $d$ has the form $Z_{\C^2}(f)\cup P$ for a polynomial $f\in \C[x,y]$ of degree $d-k$ and a finite set $P$ of size $k$.
A variety $W\subset \R^D$ is a \emph{real algebraic curve} in $\R^D$ if $\dim_\R(W)=1$, and it is a \emph{real algebraic surface} if $\dim_\R(W)=2$.
A real algebraic curve $C\subset \R^2$ can be written as $Z_{\R^2}(f)$ for a polynomial $f\in \R[x,y]$;\footnote{Here too a curve may have zero-dimensional components (isolated points), but in $\R^2$ a point $(a,b)$ is defined by a single polynomial $(x-a)^2+(y-b)^2$.} we define the \emph{degree} of $C$ to be minimum degree of such an $f$.
For convenience, we will occasionally use the notion of a \emph{semialgebraic curve} in $\R^2$, which is a subset of a real curve defined by polynomial inequalities; in particular, if we remove a finite point set from a real curve, the connected components of the remainder are semialgebraic curves.

A curve $C\subset \C^2$ is {\it irreducible} if there is an irreducible $f$ such that $C=Z_{\C^2}(f)$.
An {\it irreducible component} of an algebraic curve $C\subset \C^2$ is an irreducible algebraic curve $C'$ such that $C'\subset C$.
A curve in $\C^2$ of degree $d$ has a decomposition as a union of at most $d$ irreducible components (some of which may be points).

We also need to consider singularities of curves, but only for real curves in $\R^2$ or $\R^4$.
For a curve in $\R^D$, we define a point on the curve to be a \emph{singularity} if it does not have a neighborhood in which the curve is a real manifold of dimension one (for details see \cite[Lecture 14]{Ha} or \cite[Section 3.3]{BCR}).
A key fact that we need is that the number of singularities of a curve in $\R^2$ of degree $d$ is less than $d^2$. 
More precisely, define the \emph{branches} of a singularity in a small neighborhood to be the connected components of the curve in that neighborhood after removing the singularity.
Then the total number of branches over all singularities, for any choice of sufficiently small neighborhoods, is at most $d^2$ (see \cite[Chapter 3]{F}).
For a curve in $\R^4$, we only need the fact that the number of singularities is finite (see \cite[Proposition 3.3.14]{BCR}).

\subsection{Intersection bounds}
In the proof of our main theorem we will frequently have to bound the size of the intersection of two varieties, both over $\C$ and over $\R$.
The prototype for such intersection bounds is the following lemma. Here we consider the degree of a finite point set to be its size, so the lemma says that the intersection of two curves is either a finite set of bounded size, or a curve of bounded degree.

\begin{lemma}[B\'ezout's Inequality]\label{lem:bezout}
If $C_1$ and $C_2$ are algebraic curves in $\C^2$ or $\R^2$,
then 
$$\deg(C_1\cap C_2) \leq \deg(C_1)\cdot \deg(C_2).$$
\end{lemma}

With the right definition of degree, this inequality can be extended to varieties in $\C^D$, but in $\R^D$, the inequality may fail in this form.
Nevertheless, various cautious bounds on the number of connected components of intersections of real varieties have been proved, which can often serve the same purpose; see for instance \cite[Chapter 7]{BPR}.
We will use the following recent result of Barone and Basu \cite{BB}. 
It gives a refined bound when the defining polynomials of the variety have different degrees, which is crucial in our proofs. 
We state it in a similar way to Basu and Sombra \cite[Theorem 2.5]{BS}, with some modifications based on the more general form in \cite{BB}.
We simplify the statement of the bound somewhat using the following (non-standard) definition.

\begin{definition}\label{def:BBdegree}
Let $V := Z_{\R^D}(g_1,\ldots,g_m)$ have dimension $k_m$,
with $\deg(g_1)\leq \cdots \leq \deg(g_m)$.
Write $k_i:=\dim_\R(Z_{\R^D}(g_1,\ldots,g_i))$ and $k_0 := D$. 
We define the \emph{Barone-Basu degree} of $V$ by 
$$\degbb(V) := \prod_{i=1}^m \deg(g_i)^{k_{i-1}-k_i}.$$
\end{definition}

Note that, for example, a two-dimensional variety $Z_{\R^4}(g_1,\ldots,g_m)$ in $\R^4$ that is defined by any number of polynomials of degree at most $d$ has Barone-Basu degree $d^2$.
Indeed, we have $k_0=4$ and $k_m =2$, and either there are two $g_i$ such that $k_{i-1}-k_i = 1$, or there is one $g_i$ such that $k_{i-1}-k_i = 2$; in both cases the Barone-Basu degree comes out to $d^2$.

\begin{lemma}[Barone-Basu]\label{lem:baronebasu}
Let $V:=Z_{\R^D}(g_1,\ldots,g_m)$ with $\deg(g_1)\leq \cdots \leq \deg(g_m)$.
Let $h\in \R[x_1,\ldots,x_D]$ with $\deg(h)\geq \deg(g_m)$.
Then the number of connected components of both 
$V\cap Z_{\R^D}(h)$ and 
$V\backslash Z_{\R^D}(h)$ is 
$$O\left(\degbb(V)\cdot\deg(h)^{\dim_\R(V)}\right).$$
\end{lemma}

In the ideal case where each $k_i = D-i$, this would be a natural generalization of Lemma \ref{lem:bezout}.
On the other hand, if $\deg(g_i)\leq d$ for each $i$, we get the bound $O(d^{D-k_m}\deg(h)^{k_m})$, without any individual conditions on the $g_i$.
The fact that $h$ is arbitrary allows for the following trick to deal with more polynomials in the role of $h$: To bound the number of connected components of, say, $Z_{\R^D}(g_1,\ldots,g_m)\backslash Z_{\R^D}(h_1,h_2)$, one can simply set $h:=h_1^2+h_2^2$ and use the lemma.

Finally, we record a simple fact about the surface in $\R^4$ associated to a curve in $\C^2$.

\begin{lemma}\label{lem:complextoreal}
Let $C\subset \C^2$ be an algebraic curve of degree $d$.
Then the associated real surface $S$ in $\R^4$ is defined by two polynomials of degree at most $2d$, and $\degbb(S)\leq 4d^2$.
\end{lemma}
\begin{proof}
There is a finite set $P$ and a polynomial $f(x,y)$ of degree $d-|P|$ such that we can write $C=Z_{\C^2}(f)\cup P$. 
The real polynomials 
$$h_1(x_1,x_2,x_3,x_4) := \re f(x_1+ix_2,x_3+ix_4),~~
h_2(x_1,x_2,x_3,x_4) := \im f(x_1+ix_2,x_3+ix_4)$$
define the surface in $\R^4$ associated to $Z_{\C^2}(f)$; both have degree at most $d-|P|$.
The set $P$, viewed as a subset of $\R^4$, is defined by a polynomial $h_3$ of degree $2|P|$.
If we set $g_1 = h_1h_3$ and $g_2 = h_2h_3$, 
then we have $S = Z_{\R^4}(g_1,g_2)$, and $g_1,g_2$ have degree at most $d-|P|+2|P|\leq 2d$.

Write $k_0 := 4$, $k_1 := \dim_\R(Z_{\R^4}(g_1))$, and $k_2 := \dim_\R(Z_{\R^4}(g_1,g_2))$ as in Definition \ref{def:BBdegree}.
We clearly have $k_2=2$. Then $k_1\in \{2,3,4\}$, and, whichever it is, we get
\[\degbb(S) \leq (2d)^{k_0-k_1}\cdot (2d)^{k_1-k_2} \leq 4d^2.\]
This completes the proof.
\end{proof}

\subsection{Polynomial partitioning}
Our proof relies on the following technique introduced by Guth and Katz \cite{GK}.

\begin{lemma}[Polynomial partitioning]\label{lem:polpart}
Let $A$ be a finite subset of $\R^2$.
For any $r\in \R$ with $1\leq r\leq |A|^{1/2}$ there exists a polynomial $f\in\R[x,y]$ of degree $O(r)$
such that $\R^2\backslash Z_{\R^2}(f)$ has $O(r^2)$ connected components, each containing $O(|A|/r^2)$ points of $A$.
\end{lemma}

In the proof of Theorem \ref{thm:warmup}, we in fact used a trivial partitioning on $\R$: For $A\subset \R$ and any $1\leq r\leq |A|$, there is a subset $X\subset \R\backslash A$ of size $O(r)$ such that $\R\backslash X$ has $O(r)$ connected components, 
each containing $O(|A|/r)$ points of $A$.
Moreover, we used the fact that the points of $X$ have some ``wiggle room'', in the sense that they can be varied in some small neighborhood without affecting the partitioning property.
We now show that a point set on a real algebraic curve can be partitioned in a similar way.

Such a partitioning would not be possible for arbitrary continuous curves with self-intersections, or for algebraic curves of arbitrary degree. If we take an arbitrary point set in general position and connect any two points by a line, 
the union of the lines is an algebraic curve of high degree that cannot be partitioned with a small number of cutting points on the curve. 
However, on an algebraic curve of bounded degree $\delta$, one can control the number of self-intersections (singularities) of the curve in terms of $\delta$, and this allows us to partition it into $O(\delta^2)$ pieces.

\begin{lemma}[Partitioning a real algebraic curve]\label{lem:partoncurve}
Let $C\subset\R^2$ be an algebraic curve of degree $\delta$,
containing a finite set $A$.
Then there is a subset $X\subset C\backslash A$ of $O(\delta^2)$ points,
such that $C\backslash X$ consists of $O(\delta^2)$ connected semialgebraic curves,
each containing $O(|A|/\delta^2)$ points of $A$.
Moreover, each point $p\in X$ has an open neighborhood on $C$ such that any point of that neighborhood could replace $p$ without affecting the partitioning property.
\end{lemma}
\begin{proof}
Around every singularity $p$ of $C$, choose a sufficiently small closed ball $B_p$ with boundary circle $R_p$,
so that $B_p$ contains no other singularities of $C$, 
and no point of $A$ other than possibly $p$ itself.
We put the points of $C\cap R_p$ into $X$ for each $p$. 
For each singularity $p$, $|C\cap R_p|$ is at most the number of branches of $C$ at $p$ in the neighborhood $B_p$ (as defined at the end of Subsection \ref{subsec:defs}), and the total sum of these numbers is at most $\delta^2$.
Hence we have put at most $\delta^2$ points into $X$.
The points of $X$ are themselves not singularities, so removing a point of $X$ increases the number of connected components by at most one, since around such a point $C$ is a one-dimensional manifold.

By Lemma \ref{lem:baronebasu}, $C$ has at most $O(\delta^2)$ connected components, 
so removing the points of $X$ cuts $C$ into $O(\delta^2)$ connected semialgebraic curves. 
Each of these semialgebraic curves either contains at most one point of $A$ (a singularity),
or it is simple (i.e., it has no self-intersections).
We can cut these simple curves at a total of $O(\delta^2)$ points, so that every resulting curve contains $O(|A|/\delta^2)$ points of $A$, and no cutting point is in $A$. Adding these cutting points to $X$ completes the proof.
It should be clear that shifting the cutting points within a sufficiently small open neighborhood will not affect the proof.
\end{proof}

\section{Main bound for surfaces in \texorpdfstring{$\R^4$}{real 4-space}}
\label{sec:main}
In this section we prove our main incidence bound for points and surfaces in $\R^4$, from which we will deduce our incidence bounds for complex algebraic curves in Section \ref{sec:corcurves}.
It only applies to surfaces that are well-behaved in the following way.

\begin{definition}
A surface $S$ in $\R^4$ \emph{has good fibers} if for every $p\in \R^2$, the fibers $(p\times \R^2)\cap S$ and $(\R^2\times p)\cap S$ are finite.
\end{definition}

Note that if a curve in $\C^2$ contains no horizontal or vertical line, then its associated surface in $\R^4$ has good fibers.
Since it is easy to remove a line from a curve, this property is easily ensured.
On the other hand, for a surface $S$ in $\R^4$, the fiber $(p\times \R^2)\cap S$ may be a one-dimensional curve, which is not so easily removed.
Nevertheless, see \cite{RS} for an example of a situation where the surfaces have this property.
We also note that for surfaces with a limited set of bad fibers, the proof below might still be made to work.

Our proof uses the Guth-Katz polynomial partitioning technique from Lemma \ref{lem:polpart} in a special way that is adjusted to the Cartesian product structure.
Specifically, $\R^4$ is viewed as a product $\R^2\times \R^2$, 
and we partition each copy of $\R^2$ separately.
We first partition $\R^2$ using a curve provided by Lemma \ref{lem:polpart}, 
and then we partition that curve using Lemma \ref{lem:partoncurve}.
The partitions of the two copies of $\R^2$ are then combined into a cell decomposition of $\R^4$.

To make the bookkeeping of these partitions a bit easier to follow, we use the following terminology for our cell decomposition of $\R^4$: 
a $k$\emph{-cell} is a connected set of dimension $k$ that will be used in the final cell decomposition; a $k$\emph{-wall} is a $k$-dimensional variety that cuts out the $(k+1)$-cells, but that is itself to be decomposed into lower-dimensional cells;
a $k$\emph{-gap} is a $k$-dimensional variety that also helps to cut out the $(k+1)$-cells,
but does not contain any incidences, so does not need to be partitioned further.
To summarize: $\R^4$ is partitioned into $4$-cells by $3$-walls and $3$-gaps; each $3$-wall is then partitioned into $3$-cells by $2$-walls and $2$-gaps;
the $2$-walls are then partitioned into $2$-cells using only $1$-gaps.

In the statement that we prove here, we make the degrees-of-freedom condition a bit more flexible.
We view the set $I(\CP,\CS)$ as an \emph{incidence graph}, i.e., the bipartite graph with vertex sets $\CP$ and $\CS$, where $p\in \CP$ is connected to $S\in \CS$ if $p\in S$.
The condition that any two points are in at most $M$ surfaces can then be rephrased as $I(\CP,\CS)$ containing no complete bipartite subgraph $K_{2,M}$.
Here we weaken that condition by considering a \emph{subgraph} of $I(\CP,\CS)$; 
we show that if that subgraph contains no $K_{2,M}$, then its number of edges (denoted by $|I|$) is bounded.
This formulation is often convenient in applications; see \cite{ST,Z} for incidence bounds that are also stated in this way.

\begin{theorem}\label{thm:main}
Let $A_1$ and $A_2$ be finite subsets of $\R^2$ and $\CP := A_1\times A_2$. 
Let $\CS$ be a finite set of algebraic surfaces in $\R^4$  that have good fibers and are defined by polynomials of degree at most $d$.
Let $I\subset I(\CP,\CS)$ be an incidence subgraph containing no $K_{2,M}$.
Assume that $d^8|\CS| \leq M|\CP|^2$,
and that $|A_1|\leq|A_2|$ and $d^2|\CS|\geq M|A_2|^2/|A_1|$.
Then
$$|I| = O(d^{4/3}M^{1/3}|\CP|^{2/3}|\CS|^{2/3}).$$
\end{theorem}
\begin{proof}
Every surface $S\in \CS$ has $\degbb(S)\leq d^2$ by the remark just after Definition \ref{def:BBdegree}.
As in the proof of Theorem \ref{thm:warmup}, we partition the space, see how the varieties intersect the cells, and then use a simple counting argument.
We note that the counting is exactly as in Theorem \ref{thm:warmup}, except that the parameters $d$ and $r$ from that proof are replaced by $d^2$ and $r^2$ in this proof.

\paragraph{Partitioning.}
We partition $\R^4$ into $O(r^4)$ cells and some gaps, so that each cell contains $O(|\CP|/r^4)$ points of $\CP$, and the gaps contain no points of $\CP$.  We assume that $d^2\leq r^2\leq |A_1|$.

We use Lemma \ref{lem:polpart} to get polynomials $f_1,f_2$ of degree $r\leq |A_1|^{1/2}$ so that $C_i:=Z_{\R^2}(f_i)$ partitions $\R^2$ into $r^2$ cells, each containing $O(|A_i|/r^2)$ points of $A_i$.
Then we use Lemma \ref{lem:partoncurve} to partition $C_1$ and $C_2$, obtaining sets $X_i\subset C_i\backslash A_i$ with $|X_i| = O(r^2)$, so that $C_i\backslash X_i$ consists of $O(r^2)$ connected components, each containing $O(|A_i|/r^2)$ points of $A_i$.

The $3$-walls $C_1\times \R^2$ and $\R^2\times C_2$ partition $\R^4$ into $O(r^4)$ $4$-cells, each containing $O(|\CP|/r^4)$ points of $\CP$.
The $3$-wall $C_1\times \R^2$ 
is partitioned by the $2$-wall $C_1\times C_2$, combined with the $2$-gap $X_1\times \R^2$;
similarly, $\R^2\times C_2$ is partitioned by the $2$-wall $C_1\times C_2$ and the $2$-gap $\R^2\times X_2$.
Thus the $3$-walls are partitioned into $O(r^4)$ $3$-cells, each containing $O(|\CP|/r^4)$ points of $\CP$.
The gaps are not partitioned further. 
The $2$-wall $C_1\times C_2$ is partitioned by the $1$-gaps $X_1\times C_2$ and $C_1\times X_2$,
again resulting in $O(r^4)$ cells, each containing $O(|\CP|/r^4)$ points of $\CP$.
This completes the partitioning.
Altogether there are $O(r^4)$ cells, each containing $O(|\CP|/r^4)$ points of $\CP$.

\paragraph{Intersections.}
We now show that any surface $S\in \CS$ intersects $O(d^2r^2)$ of the $O(r^4)$ cells, and we do this separately for the $4$-cells, $3$-cells, and $2$-cells.

\emph{$4$-cells:} 
The $4$-cells are cut out by the $3$-wall $(C_1\times \R^2)\cup (\R^2\times C_2) = Z_{\R^4}(f_1f_2)$. 
To get an upper bound on the number of $4$-cells intersected by $S$, 
we want an upper bound on the number of connected components of $S\backslash Z_{\R^4}(f_1f_2)$.
We apply Lemma \ref{lem:baronebasu} to deduce that $S\backslash Z_{\R^4}(f_1f_2)$ has 
\[O(\degbb(S)\cdot \deg(f_1f_2)^{\dim_\R(S)})=O(d^2\cdot (2r)^2)\]
connected components;
here $\degbb(S)\le d^2$ by assumption, and the condition of Lemma \ref{lem:baronebasu} 
(that the degree of $f_1f_2$ is at least the degree of the polynomials defining $S$) follows from the assumption $d^2\leq r^2$.
This means that $S$ intersects $O(d^2r^2)$ of the $4$-cells.

\emph{$3$-cells:} Set $S_1:=S\cap (C_1\times \R^2)$. 
Note that $\dim_\R(S_1)\le 1$, because for any  $p\in \R^2$ the fiber $(p\times \R^2)\cap S$ is finite, since $S$ has good fibers.
If $\dim_\R(S_1)=0$, then by Lemma \ref{lem:baronebasu} $S_1 = S \cap Z_{\R^4}(f_1)$ consists of $O(d^2\cdot r^2)$ points, so it intersects at most that many cells.
Hence we can assume $\dim_\R(S_1)= 1$.
To see how many $3$-cells inside $C_1\times \R^2$ are intersected by $S_1$,
we separately consider its intersection with the $2$-wall $C_1\times C_2$, and with the $2$-gap $X_1\times \R^2$.

The fact that $\dim_\R(S_1)=1$ implies that $\degbb(S_1) = O(d^2r)$.
Therefore, by Lemma \ref{lem:baronebasu}, $S_1\backslash (C_1\times C_2) =S_1\backslash Z_{\R^4}(f_2)$ has 
\[O(\degbb(S_1)\cdot \deg(f_2)^{\dim_{\R}(S_1)}) = O(d^2r\cdot r)\]
connected components.
Hence the wall $C_1\times C_2$ cuts $S_1$ into $O(d^2r^2)$ connected semialgebraic curves.

Now consider the gap $X_1\times \R^2$.\footnote{Note that $X_1\times \R^2$ is defined by a polynomial $g$ of degree $2|X_1|=O(r^2)$.
Thus, applying Lemma \ref{lem:baronebasu} to $S_1\backslash Z_{\R^4}(g)$ gives $O(d^2r\cdot r^2)$, which is too large. 
This is why we need a more refined argument, using the specific nature of $X_1\times \R^2$.}
For $p\in X_1$,
we have $S_1\cap (p\times \R^2)\subset S\cap (p\times \R^2)$, and $S\cap (p\times \R^2)$ is finite, again because $S$ has good fibers.
Since we can write $p\times \R^2=Z_{\R^4}((x_1-p_x)^2+(x_2-p_y)^2)$, it follows from Lemma \ref{lem:baronebasu} that $|S\cap (p\times \R^2)|=O(d^2\cdot 2^2)$.
Thus the curve $S_1$ has 
\[|S_1\cap (X_1\times \R^2)| =O(  d^2\cdot |X_1|) = O(d^2r^2)\] points of intersection with this gap.
Moreover, using the ``wiggle room'' for the points in $X_1$ mentioned in Lemma \ref{lem:partoncurve},
and the fact that $S_1$ has finitely many singularities, 
we can assume that none of the points in $S_1\cap (X_1\times \R^2)$ is a singularity of $S_1$.
Hence, removing such a point increases the number of connected components by at most one 
(which would not quite be true at a singularity). 
Since the wall $C_1\times C_1$ cuts $S_1$ into $O(d^2r^2)$ connected components, and we remove $O(d^2r^2)$ further points, 
it finally follows that $S_1$ intersects $O(d^2r^2)$ of the $3$-cells inside $C_1\times \R^2$.
Note that $X_1$ should be chosen so that $X_1\times \R^2$ avoids the singularities of $S_1$ for all $S\in \CS$ simultaneously, 
but this is possible since there are finitely many points to avoid, while there is infinite wiggle room.

A symmetric argument gives the same bounds for $S_2:= S\cap (\R^2\times C_2)$, so altogether we get that $S$ intersects $O(d^2r^2)$ of the $3$-cells inside $\R^2\times C_2$.

\emph{$2$-cells:} 
Set $S_3 := S\cap (C_1\times C_2)$.
The $2$-wall $C_1\times C_2$ is partitioned by the $1$-gaps $X_1\times C_2$ and $C_1\times X_2$.
As above we have $|S_3\cap (p\times C_2)|=O(d^2)$ for $p\in X_1$, so we get $|S_3\cap (X_1\times C_2)|=O(d^2r^2)$ and similarly $|S_3\cap (C_1\times X_2)|=O(d^2r^2)$.
Finally, we can write $S_3 = S\cap Z_{\R^4}(f_1^2+f_2^2)$, so $S_3$ has $O(d^2\cdot (2r)^2)$ connected components.
Again each cut increases the number of connected components by at most one, so altogether $S_3$ intersects $O(d^2r^2)$ of the $2$-cells. 

\paragraph{Counting.}
Let $I_1$ be the subset of incidences $(p,S)\in I$ such that $(p,S)$ is the only incidence of $S$ from $I$ in the cell containing $p$,
and let $I_2$ be the subset of incidences $(p,S)\in I$ such that $S$ has at least one other incidence from $I$ in the cell that contains $p$.
The fact that a surface from $\CS$ intersects $O(d^2r^2)$ cells implies
$$|I_1| = O\left(d^2r^2|\CS|\right).$$
On the other hand, given two points $p_1,p_2$ in one cell, there are by assumption fewer than $M$ surfaces $S\in \CS$ such that $(p_1,S),(p_2,S)\in I$.
Thus we have
$$|I_2|
= O\left( r^4\cdot M\cdot \left(\frac{|\CP|}{r^4}\right)^2\right)
= O\left(M\cdot\frac{|\CP|^2}{r^4}\right).$$
Choosing  $r^6 := \frac{M}{d^2}\frac{|\CP|^2}{|\CS|}$ gives $|I(\CP, \CS)| = O\left(d^{4/3}M^{1/3}|\CP|^{2/3}|\CS|^{2/3}\right)$.
We need to ensure that $d^2\leq r^2\leq |A_1|$; this follows from the two assumptions of the theorem, by the same calculation as in the proof of Theorem \ref{thm:warmup} (with $d$ and $r$ replaced by $d^2$ and $r^2$). 
\end{proof}


\section{Corollaries for surfaces in \texorpdfstring{$\R^4$}{real 4-space}}\label{sec:corsurfaces}

We now deduce some more practical corollaries of Theorem \ref{thm:main} for surfaces in $\R^4$, without the awkward conditions on the sizes of the sets of points and surfaces.
To remove these conditions we use the K\H ov\'ari-S\'os-Tur\'an theorem, a commonly used tool in incidence geometry.
It gives a bound on the number of edges in a graph not containing a complete bipartite graph $K_{s,t}$; see Bollob\'as \cite[Theorem IV.10]{B} for the version stated here.

\begin{lemma}\label{lem:KST}
Let $G\subset X\times Y$ be a bipartite graph.
Suppose that $G$ contains no $K_{s,t}$, i.e., for any $s$ vertices in $X$, there are fewer than $t$ vertices in $Y$ connected to both.
Then the number of edges of $G$ is bounded by
$$ O(t^{1/s}|X||Y|^{1-1/s} + s|Y|). $$
\end{lemma}

We now use this lemma to obtain a more convenient version of Theorem \ref{thm:main}.
See \cite{RS} for an application of this corollary.

\begin{corollary}\label{cor:surfaces1}
Let $A$ and $B$ be finite subsets of $\R^2$ with $|A|= |B|$, and $\CP = A\times B\subset \R^4$.
Let $\CS$ be a finite set of surfaces in $\R^4$  that have good fibers and are defined by polynomials of degree at most $d$.
Let $I\subset I(\CP,\CS)$ be an incidence subgraph containing no $K_{2,M}$ or $K_{M,2}$.
Then
 $$|I| 
= O_{d,M}\left(|\CP|^{2/3}|\CS|^{2/3}
+ |\CP|
+ |\CS|
\right).$$
\end{corollary}
\begin{proof}
If $d^{-2}M|\CP|^{1/2}\le |\CS|\le d^{-8}M|\CP|^2$, we can apply Theorem \ref{thm:main} directly, which results in the first term of the bound.
If $|\CS|>d^{-8}M|\CP|^2$, then we can apply Lemma \ref{lem:KST} with $X := \CP$ and $Y := \CS$ to get
\[ |I| = O_M(|\CP||\CS|^{1/2} + |\CS| ) = O_M(|\CS|).\]
On the other hand,
if $|\CS| < d^{-2}M|\CP|^{1/2}$,
then Lemma \ref{lem:KST} with $X := \CS$ and $Y := \CP$ gives
\[ |I| = O_M(|\CS||\CP|^{1/2} + |\CP| ) = O_M(|\CP|).\]
Combining these bounds proves the corollary.
\end{proof}

Next we prove a version of Theorem \ref{thm:main} where the condition on the excluded complete bipartite subgraph is weakened in a different way: Instead of requiring every two points to lie in a bounded number of surfaces, we only require this for any $s$ points.
Such a bound was given for curves in \cite{PS};
see \cite{Z} for a similar statement for surfaces in $\R^4$.
To prove it we only have to modify the counting step in the proof of Theorem \ref{thm:main}.

\begin{theorem}\label{thm:noKst}
Let $A_1$ and $A_2$ be finite subsets of $\R^2$ and $\CP := A_1\times A_2$. 
Let $\CS$ be a finite set of algebraic surfaces  in $\R^4$ that have good fibers and are defined by polynomials of degree at most $d$.
Let $I\subset I(\CP,\CS)$ be an incidence subgraph containing no $K_{s,t}$.
Assume that $|\CS|\leq d^{-(4s-2)}|\CP|^s$,
and that $|A_1|\leq|A_2|$ and $|\CS|\geq |A_1|^{1-s}|A_2|^s $.
Then
$$|I| = O_{d,s,t}(|\CP|^{\frac{s}{2s-1}}|\CC|^{\frac{2s-2}{2s-1}}).$$
\end{theorem}
\begin{proof}
As said, we reuse the partitioning and intersection steps from the proof of Theorem \ref{thm:main}, and we jump in at the counting step.

Let $I_1$ be the subset of incidences $(p,S)\in I$ such that $S$ has at most $s-1$ incidences from $I$ in the cell that $p$ lies in.
Let $I_2$ be the subset of incidences $(p,S)\in I$ such that $C$ has at least $s$ incidences from $I$ in the cell that $p$ lies in.
The fact that a surface from $\CS$ intersects $O(d^2r^2)$ cells implies
$$|I_1| = O_{d,s}\left(r^2|\CS|\right).$$
On the other hand, given $s$ points in one cell, there are by assumption fewer than $t$ surfaces $S\in \CS$ containing all $s$ points.
Thus we have
$$|I_2|
= O\left( r^4\cdot t\cdot \left(\frac{|\CP|}{r^4}\right)^s\right)
= O\left(t\cdot\frac{|\CP|^s}{r^{4s-4}}\right).$$
Setting $r^{4s-2} = \frac{|\CP|^s}{|\CS|}$ gives
$$|I| = O_{d,s,t}(|\CP|^{\frac{s}{2s-1}}|\CC|^{\frac{2s-2}{2s-1}}).$$

We need to ensure that $d^2\leq r^2\leq |A_1|$.
The assumption that $|\CS|\leq d^{-(4s-2)}|\CP|^s$ gives $r^{4s-2} \geq d^{4s-2}$, 
and the assumption that $|\CS|\geq |A_1|^{1-s}|A_2|^s $ gives $r^{4s-2}\leq |A_1|^{2s-1}$.
\end{proof}

Again, we can prove a version with more practical conditions.

\begin{corollary}\label{cor:surfaces2}
Let $A$ and $B$ be finite subsets of $\R^2$ with $|A|= |B|$, and $\CP := A\times B\subset \R^4$.
Let $\CS$ be a finite set of surfaces in $\R^4$ that have good fibers and are defined by polynomials of degree at most $d$.
Let $I\subset I(\CP,\CS)$ be an incidence subgraph containing no $K_{s,t}$ or $K_{t,2}$.
Then
 $$|I| 
= O_{d,s,t}\left( |\CP|^{\frac{s}{2s-1}}|\CS|^{\frac{2s-2}{2s-1}} + |\CP| + |\CS|
\right).$$
\end{corollary}
\begin{proof}
If $|\CP|^{1/2}\leq |\CS|\leq d^{-(4s-2)}|\CP|^s$, we can apply Theorem \ref{thm:main} directly, which results in the first term of the bound.
If $|\CS|>d^{-(4s-2)}|\CP|^s$, then Lemma \ref{lem:KST} gives $|I| = O_{d,s,t}(|\CS|)$,
while if $|\CS| < |\CP|^{1/2}$,
then Lemma \ref{lem:KST} gives $|I| = O_M(|\CP|)$.
Combining these bounds proves the corollary.
\end{proof}


\section{Corollaries for curves in \texorpdfstring{$\C^2$}{the complex plane}}\label{sec:corcurves}
In this section we deduce several incidence bounds for complex algebraic curves from Theorem \ref{thm:main}, including Theorem \ref{thm:mainintro}.
There are many different ways to vary these statements, and we certainly do not cover all combinations, but we focus on those that have turned out useful in applications (see Section \ref{sec:applications} and \cite{RSZ, VZ}).

We make some effort to determine the dependence of the bounds on the parameters $d$ and $M$, because this is of interest in the applications \cite{VZ, RSZ}.
In the incidence bounds for curves in $\C^2$ that were proved in \cite{ST,Z,SZ}, determining this dependence seems challenging.

The following corollary is our first practical incidence bound for curves in $\C^2$.
Note that the second term in the bound is a bit awkward, but this seems unavoidable when $|A|\neq |B|$.

\begin{corollary}\label{cor:practical}
Let $A$ and $B$ be finite subsets of $\C$ with $|A|\leq|B|$, and let $\CP := A\times B$.
Let $\CC$ be a finite set of algebraic curves in $\C^2$ of degree $d$ such that no two have a common component.
Let $I\subset I(\CP,\CC)$ be an incidence subgraph containing no $K_{2,M}$.
Then
$$|I| 
= O(d^{4/3}M^{1/3}|\CP|^{2/3}|\CC|^{2/3}
+ d^{-1}M|A|^{-1/2}|B|^{5/2}
 + d^4|\CC|).$$
\end{corollary}
\begin{proof}
Let $I_1$ be the subset of incidences $(p,C)\in I$ such that $p$ lies on a horizontal or vertical line contained in $C$.
Since the curves have no common components, each horizontal or vertical line occurs at most once, and any point is contained in at most two such lines, so
$$|I_1|\leq 2|\CP|.$$

Let $I_2$ be the subset of incidences $(p,C)\in I$ such that $p$ does not lies on a horizontal or vertical line contained in $C$.
Let $\CC^*$ be the set of curves obtained by removing all the horizontal and vertical lines from the curves in $\CC$; we have $|\CC^*|\leq |\CC|$.
We can view $I_2$ as a subgraph of the incidence graph $I(\CP,\CC^*)$.
The fact that the curves in $\CC^*$ contain no horizontal or vertical lines implies that the associated surfaces in $\R^4$ have good fibers,
and by Lemma \ref{lem:complextoreal} the surfaces are defined by polynomials of degree at most $2d$.
Hence we can apply Theorem \ref{thm:main} to obtain 
\[|I_2| = O(d^{4/3}M^{1/3}|\CP|^{2/3}|\CC|^{2/3}),\] 
unless we have $d^8|\CC^*| > M|\CP|^2$ or $d^2|\CC^*|< M|B|^2/|A|$.

Suppose that $d^8|\CC^*| > M|\CP|^2$.
Since $I$ contains no $K_{2,M}$, we can use Lemma \ref{lem:KST} to get
$$|I|=O(M^{1/2}|\CP||\CC^*|^{1/2} +|\CC^*|)
=O(d^4|\CC|).$$

Suppose that $d^2|\CC^*|<M|B|^2/|A|$.
Because the curves do not have common components, any two curves intersect in at most $d^2$ points by B\'ezout's Inequality (Lemma \ref{lem:bezout}).
Thus $I_2$ contains no $K_{d^2+1,2}$, so by Lemma \ref{lem:KST} we have
$$|I|\leq |I(\CP,\CC^*)|=O((d^2)^{1/2}|\CC^*||\CP|^{1/2} + |\CP|)
=O(d^{-1}M|A|^{-1/2}|B|^{5/2}).$$
Combining these bounds finishes the proof.
\end{proof}

With a little more work, we can remove the condition that no two curves have a common component, with almost no effect on the bound; this is Theorem \ref{thm:mainintro}.
Note that we lose the flexibility of allowing a subgraph of the incidence graph. Indeed, the curves could all share a common component, and on that component the incidence subgraph could be any bipartite graph $K_{2,M}$, which need not satisfy the desired bound.

\begin{corollary}\label{cor:practical2}
Let $A$ and $B$ be finite subsets of $\C$ with $|A|= |B|$, and $\CP := A\times B$.
Let $\CC$ be a finite set of algebraic curves in $\C^2$ of degree at most $d$, such that any two points of $\CP$ are contained in at most $M$ curves of $\CC$.
Then
 $$|I(\CP, \CC)| 
= O\left(d^{4/3}M^{1/3}|\CP|^{2/3}|\CC|^{2/3}
+M(\log M+\log d)|\CP|
+ d^4|\CC|
\right).$$
\end{corollary}
\begin{proof}
We have to deal with horizontal or vertical lines in the curves, and with the case $d^2|\CC|< M|\CP|^{1/2}$;
the other cases can be treated as in Corollary \ref{cor:practical}.

Let $\CC_1$ be the \emph{multiset} of horizontal and vertical lines contained in curves of $\CC$,
and let $\CC_2$ be the curves that remain after these lines have been removed.
In total there are at most $d|\CC|$ lines in $\CC_1$ (counted with multiplicity).
The lines that contain at most one point of $\CP$ together give at most $d|\CC|$ incidences.
The lines that contain at least two points of $\CP$ have multiplicity at most $M$ by assumption, 
so a point on such a line is contained in at most $2M$ such lines (counted with multiplicity), resulting in at most $2M|\CP|$ incidences.
Hence
$$|I(\CP,\CC_1)|=O(M|\CP|+d|\CC|).$$

Now suppose $d^2|\CC|< M|\CP|^{1/2}$.
We split each curve in $\CC_2$ into its at most $d$ irreducible components.
The components that contain at most one point of $\CP$ give altogether at most $d|\CC|$ incidences.
Let $\CC^*$ be the \emph{multiset} of components that contain at least two points of $\CP$.
A curve in $\CC^*$ has multiplicity at most $M$ by assumption.

Let $\CC_{ij}$ be the \emph{set} of curves in $\CC^*$ that have multiplicity between $2^i$ and $2^{i+1}$
 and degree between $2^j$ and $2^{j+1}$.
The sum of all the degrees of all the components of the curves in $\CC$ is at most $d|\CC|$, 
so the number of curves of degree at least $2^j$ that occur with multiplicity at least $2^i$ is bounded by $d|\CC|/2^{i+j}$.
Thus
$$|\CC_{ij}|
\leq d|\CC|/2^{i+j}
\leq d^{-1}M|\CP|^{1/2}/2^{i+j},$$ and two distinct curves in $\CC_{ij}$ intersect in at most $2^{j+1}\cdot 2^{j+1} = 4\cdot 2^{2j}$ points by Lemma \ref{lem:bezout}. 
Hence the incidence graph $I(\CP, \CC_{ij})$ contains no $K_{(4\cdot2^{2j}+1), 2}$, so Lemma \ref{lem:KST} gives
$$
|I(\CP,\CC_{ij})|
= O\left((2^{2j})^{1/2}(d^{-1}M|\CP|^{1/2}/2^{i+j})|\CP|^{1/2} + |\CP|\right)
= O\left(2^{-i}d^{-1}M|\CP|+|\CP|\right).
$$
Therefore,
\begin{align*}
|I(\CP,\CC^*)|
&\leq \sum_{i=1}^{\log M}\sum_{j=1}^{\log d}2^{i+1} I(\CP,\CC_{ij})
=O\left(\sum_{i=1}^{\log M}\sum_{j=1}^{\log d}d^{-1} M|\CP|+2^i|\CP|\right)\\
&=O\left(d^{-1}M\log M\log d|\CP|+
M\log d |\CP|
\right)
=O\left(M(\log M+\log d)|\CP|
\right)
.
\end{align*}
Together with $|I(\CP,\CC_2)| = |I(\CP,\CC^*)|+O(d|\CC|)$ this completes the proof.
\end{proof}

Finally, we state a curve version of Corollary \ref{cor:surfaces2}.

\begin{corollary}
Let $A$ and $B$ be finite subsets of $\C$ with $|A|= |B|$, and $\CP := A\times B$.
Let $\CC$ be a finite set of algebraic curves in $\C^2$ of degree at most $d$, such that any $s$ points of $\CP$ are contained in at most $t$ curves of $\CC$.
Then
 $$|I(\CP, \CC)| 
= O_{d,s,t}\left( |\CP|^{\frac{s}{2s-1}}|\CC|^{\frac{2s-2}{2s-1}} + |\CP| + |\CC|
\right).$$
\end{corollary}
\begin{proof}
If $|\CP|^{1/2}\leq |\CC|\leq d^{-(4s-2)}|\CP|^s$, we can apply Theorem \ref{thm:main} to the surfaces associated to the curves, which results in the first term of the bound.
If $|\CC|>d^{-(4s-2)}|\CP|^s$, then Lemma \ref{lem:KST} gives
\[ |I(\CP,\CC)| = O_{s,t}(|\CP||\CC|^{1-1/s} + |\CC| ) = O_{d,s,t}(|\CC|).\]
If $|\CC| < |\CP|^{1/2}$,
then arguing  as in the proof of Corollary \ref{cor:practical2} gives $|I(\CP,\CC)| = O_{d,t}(|\CP|).$
\end{proof}


\section{Applications}\label{sec:applications}
We now show several examples of applications in which the assumption that the point set is a Cartesian product is satisfied.
We do not work out these applications in the greatest generality here, but merely give some samples that should illustrate the usefulness of our bounds.

\paragraph{Rich transformations.}
Elekes \cite{E02,E11} introduced various questions of the following form: \emph{Given a group $G$ of transformations on some set $X$ and an integer $k$, 
 what is the maximum size of }
 $$R_k(S):=\{\varphi\in G: |\varphi(S)\cap S|\geq k\}$$
\emph{for a finite subset $S\subset X$?}
The work of Guth and Katz \cite{GK} involved this question for $X := \R^2$ and $G$ the group of Euclidean isometries of $\R^2$.
Solymosi and Tardos \cite{ST} gave the bound $|R_k(A)|=O(|A|^4/k^3)$ when $X:=\C$ and $G$ is the group of linear transformations from $\C$ to $\C$, 
and the bound $|R_k(A)|=O(|A|^6/k^5)$ when $X:=\C$ and $G$ is the group of M\"obius transformations from $\C$ to $\C$.

We give one example to illustrate how our incidence bound can be used for this type of problem.
We consider the group of M\"obius transformations $(cz+a)/(dz+b)$ for which $c=0, d=1$; i.e., the \emph{inversion transformations}.

\begin{theorem}
Let $A\subset \C$ be finite and let $R_k(A)$ be the set of inversion transformations $\varphi_{ab}(z) = a/(z+b)$, with $a,b\in \C$ and $a\neq 0$, for which $|\varphi_{ab}(A)\cap A|\geq k$.
Then 
$$|R_k(A)| = O\left(\frac{|A|^4}{k^3}\right).$$
\end{theorem}
\begin{proof}
Let $\CP :=A\times A$.
Define $C_{ab} := Z_{\C^2}(y(x+b)-a)$ and set $\CC := \{C_{ab}: \varphi_{ab}\in R_k(A)\}$. 
Then for every $C_{ab}\in \CC$ we have $|C_{ab}\cap \CP|\geq k$.

The curves $C_{ab}$ are clearly distinct.
Suppose that two points $(x,y),(x',y')\in \C^2$ lie on the curve $C_{ab}$. Then we have $y(x+b) =a=y'(x'+b)$, so
$$(y-y')b = y'x'-yx.$$
If $y\neq y'$, then $b$ is determined by this equation, and $a$ is determined by $a = y(x+b)$.
If $y' = y\neq 0$, then we have $x'=x$, a contradiction. 
If $y =0$, we would have $a = 0$, also a contradiction.
Thus at most two curves $C_{ab}$ pass through any two points $(x,y),(x',y')$.

Corollary \ref{cor:practical2} (or Theorem \ref{thm:mainintro}) then gives
$$k\cdot |\CC| \leq |I(\CP, \CC)|=O((|A|^2)^{2/3}|\CC|^{2/3}+|A|^2 +|\CC|).$$
This implies $|R_k(A)|=|\CC|=O\left(|A|^4/k^3\right)$.
\end{proof}

\paragraph{Elekes-Nathanson-Ruzsa-type problems.}
In \cite{ENR}, Elekes, Nathanson, and Ruzsa considered generalizations of sum-product inequalities over $\R$.
Their proofs converted these problems into incidence problems for points and curves over $\R$, with the point set being a Cartesian product. So these problems are well-suited to our incidence bounds over $\C$.

For instance, one of the main results of \cite{ENR} stated that if $f:\R\to\R$ is a convex function and $A\subset \R$ a finite set, then 
$$\max\{|A+A|,|f(A)+f(A)|\} = \Omega(n^{5/4})~~~\text{and}~~~|A+f(A)|=\Omega(n^{5/4}),$$
where $f(A) := \{f(a):a\in A\}$.
For a rational function $f\in\R(x)$, the same bounds can be deduced by splitting up the graph of $f$ into convex and concave pieces (the number of which is bounded in terms of the degree of $f$).
Over $\C$, it is not clear what the analogue of a convex function would be, but for polynomials or rational functions, these bounds could be generalized to $\C$.
We do this for the specific function $f(x)=1/x$, thereby solving Problem 2.10 in Elekes's survey \cite{E02}.

\begin{theorem}
Let $A\subset \C$ be finite and write $1/A := \{1/a:a\in A\}$.
Then
$$\max\{|A+A|, |1/A+1/A|\}=\Omega(n^{5/4})~~~\text{and}~~~|A+1/A|=\Omega(n^{5/4}).$$
\end{theorem}
\begin{proof}
Set $\CP := (A+A)\times(1/A+1/A)$, 
$$C_{ab} := Z_{\C^2}((x-a)(y-1/b) - 1),$$
and $\CC := \{C_{ab}: a,b\in A\}$. 
The curve $C_{ab}$ equals the graph $y = 1/(x-a) + 1/b$.
Then each of the $|A|^2$ curves $C_{ab}$ has $|C_{ab}\cap \CP|\geq |A|$, since for every $a'\in A$ we have $x=a+a'\in A+A$ and 
$$\frac{1}{x-a} + \frac{1}{b} = \frac{1}{a'}+\frac{1}{b} =y\in 1/A+1/A,$$
so $(x,y)\in C_{ab}$.

We now check that the curves in $\CC$ meet the conditions of Corollary \ref{cor:practical2}.
Suppose that two points $(x,y),(x',y')\in \C^2$ lie on the curve $C_{ab}$. Then $x\neq a,x'\neq a$. We have 
$$y - y' = \frac{1}{x-a}-\frac{1}{x'-a},$$
so $(y-y') (x-a)(x'-a) = x'-x$.
This implies $y \neq y'$, since otherwise we would also get $x' =x$.
Then we get
$$a^2 - (x+x')a +\left(xx' - \frac{x-x'}{y-y'}\right)=0.$$
At most two $a$ satisfy this equation, and $a$ determines $b$ by $y = 1/(x-a)-1/b$.
Thus at most two curves $C_{ab}$ pass through the points $(x,y),(x',y')$.

By Corollary \ref{cor:practical2}, we get (the second and third term have no effect)
$$|A|\cdot |A|^2\leq |I(\CP, \CC)| 
=O\left((|A|^2)^{2/3}(|A+A|\cdot|1/A+1/A|)^{2/3}\right).$$
This gives 
$|A+A|\cdot|1/A+1/A|=\Omega(|A|^{5/2})$.

For the second statement, we define $C^*_{ab} := Z((x-1/a)(y -b) - 1)$, which is $|A|$-rich on $(A+1/A)\times (A+1/A)$.
The conditions of Corollary \ref{cor:practical2} can be checked in a similar way, so
$$|A|^3=O\left((|A|^2)^{2/3}(|A+1/A|^2)^{2/3}\right), $$
which gives $|A+1/A|=\Omega(|A|^{5/4})$.
\end{proof}

\paragraph{Elekes-R\'onyai-type problems.}
Elekes and R\'onyai \cite{ER} introduced another class of questions that lead to incidence problems on Cartesian products in a natural way.
The strongest result in this direction was recently obtained by Raz, Sharir, and Solymosi in \cite{RSS}, and it states the following. 
Let $f(x,y)\in \R[x,y]$ be a polynomial of constant degree, let $A,B\subset \R$ with $|A|=|B|=n$, and write $f(A,B) := \{f(a,b):a\in A, b\in B\}$.
Then we have $|f(A,B)|=\Omega(n^{4/3})$, unless $f$ is of the form $g(h(x)+k(y))$ or $g(h(x)\cdot k(y))$, with $g,h,k\in \R[z]$.
In other words, $f$ is an ``expander'' unless it has a special form.
A generalization of this statement is proved in \cite{RSZ}, using our Theorem \ref{thm:mainintro}.

Again, the typical approach to these problems is by converting them into incidence problems between points and curves, with the points forming a Cartesian product.
The general analysis is considerably more difficult; in particular, the curves can actually have many common components, and one needs to show that they do not have too many common components, unless $f$ has a special form.

To illustrate how our incidence bounds can extend such results to $\C$, we establish a simple case, where the polynomial does not have the special form.
Moreover, this case is a nice geometric question. It was first considered in \cite{ER}, and the real equivalent of the bound below was obtained by Sharir, Sheffer, and Solymosi \cite{SSS}, whose proof we follow here.

Consider the ``Euclidean distance'' defined by 
$D(p,q) := (p_x-q_x)^2+(q_x-q_y)^2$ for $p=(p_x,p_y),q=(q_x,q_y)\in\C^2$, 
and write $D(A,B) :=\{D(a,b):a\in A, b\in B\}$.

\begin{theorem}
Let $L_1,L_2$ be two lines in $\C^2$, 
and $A\subset L_1, B\subset L_2$ with $|A|=|B|=n$.
Then 
$$|D(A,B)|=\Omega(n^{4/3}),$$
unless $L_1$ and $L_2$ are parallel or orthogonal.
\end{theorem}
\begin{proof}
If the lines are not parallel or orthogonal, we can assume that $L_1$ is the $x$-axis, and that $L_2$ contains the origin and is not vertical. 
Then the lines can be parametrized by $p(x)=(x,0)$ and $q(y) = (y,my)$, for some $m\in \C\backslash \{0\}$,
so that the distance is given by
$$f(x,y) := D(p(x),q(y)) = (x-y)^2 + m^2y^2.$$
We will show that the polynomial $f$ is an expander in the sense of Elekes and R\'onyai.

Set $\CP := A\times A$,
\[C_{bb'}:= Z_{\C^2}(f(x,b) - f(y,b')),\]
and $\CC := \{C_{bb'}: b,b'\in B\}$.
The equation of $C_{bb'}$ is 
\[(x-b)^2-(y-b')^2 = m^2(b'^2-b^2),\]
which defines a hyperbola, unless $b=b'$.
The curves of the form $C_{bb}$ have altogether at most $n^2$ incidences, so we can safely ignore them.
A quick calculation shows that any two points are contained in at most two hyperbolas $C_{bb'}$ with $b\neq b'$.
Thus, by Corollary \ref{cor:practical2}, we have
$$|I(\CP,\CC)|=O\left((|A|^2)^{2/3}(|B|^2)^{2/3}+|A|^2+|B|^2\right)=O(n^{8/3}). $$

Writing $f^{-1}(c):=\{(a,b)\in A\times B: f(a,b) = c\}$ and using Cauchy-Schwarz gives
\begin{align*}
|I(\CP,\CC)| +  n^2 &\geq  |\{(a,b,a',b')\in (A\times B)^2: f(a,b) = f(a',b')\}|\\
&= \sum_{c\in f(A,B)} |f^{-1}(c)|^2
\geq \frac{1}{|f(A,B)|} \left(\sum_{c\in f(A,B)} |E_c|\right)^2
= \frac{n^4}{|f(A,B)|}.
\end{align*}
Therefore $|f(A,B)| =\Omega( n^4/I(\CP,\CC))= \Omega(n^{4/3})$.
\end{proof}

\end{document}